\documentclass[12pt,a4paper]{article}
\usepackage[a4paper,margin=2cm]{geometry}

\usepackage{amsmath}
\usepackage{amsthm}
\usepackage{subcaption}
\usepackage{amssymb}
\usepackage{enumitem}
\usepackage{graphicx}
\usepackage{nccmath}
\usepackage{setspace}
\usepackage[colorlinks=true, allcolors=blue]{hyperref}
\usepackage{enumerate}
\usepackage{multirow}
\usepackage{float}
\usepackage{bm}
\usepackage[dvipsnames]{xcolor}

\newtheorem{theorem}{Theorem}

\newtheorem{lemma}[theorem]{Lemma}

\theoremstyle{definition}

\def \mod#1{{\:({\rm mod}\ #1)}}
\def \Z{\mathbb{Z}}

\def \A{\mathcal{A}}
\def \B{\mathcal{B}}
\def \T{\mathcal{T}}
\def \D{\mathcal{D}}
\def \G{\mathcal{G}}

\def \leq {\leqslant}
\def \geq {\geqslant}

\let\oldproofname=\proofname
\renewcommand{\proofname}{\textup{\textbf{\oldproofname}}}

\title{Smaller embeddings of partial $k$-star decompositions}
\author{Ajani De Vas Gunasekara \qquad  Daniel Horsley\\[2mm]
School of Mathematics, Monash University, Victoria 3800, Australia}
\date{}

\begin{document}

\maketitle

\begin{abstract}
A $k$-star is a complete bipartite graph $K_{1,k}$. For a graph $G$, a \emph{$k$-star decomposition of $G$} is a set of $k$-stars in $G$ whose edge sets partition the edge set of $G$. If we weaken this condition to only demand that each edge of $G$ is in at most one $k$-star, then the resulting object is a \emph{partial $k$-star decomposition of $G$}. An embedding of a partial $k$-star decomposition $\A$ of a graph $G$ is a partial $k$-star decomposition $\B$ of another graph $H$ such that $\A \subseteq \B$ and $G$ is a subgraph of $H$. This paper considers the problem of when a partial $k$-star decomposition of $K_n$ can be embedded in a $k$-star decomposition of $K_{n+s}$ for a given integer $s$. We improve a result of Noble and Richardson, itself an improvement of a result of Hoffman and Roberts, by showing that any partial $k$-star decomposition of $K_n$ can be embedded in a $k$-star decomposition of $K_{n+s}$ for some $s$ such that $s < \frac{9}{4}k$ when $k$ is odd and $s < (6-2\sqrt{2})k$ when $k$ is even. For general $k$, these constants cannot be improved. We also obtain stronger results subject to placing a lower bound on $n$.

\end{abstract}

\section{Introduction}

A \emph{$k$-star decomposition} of a graph $G$ is a collection of copies of $K_{1,k}$ in $G$ such that each edge of $G$ is in exactly one copy. If we weaken this condition to demand that each edge of $G$ is in at most one copy, then the resulting object is a \emph{partial $k$-star decomposition}. An \emph{embedding} of a partial $k$-star decomposition $\A$ of a graph $G$ is a partial $k$-star decomposition $\B$ of another graph $H$ such that $\A \subseteq \B$ and $G$ is a subgraph of $H$. The \emph{leave} of a partial $k$-star decomposition of $G$ is the graph $L$ having vertex set $V(G)$ and edge set comprising all edges of $G$ that are not in a $k$-star in the decomposition.

The problem of determining when a graph has a decomposition into $k$-stars has been thoroughly investigated. An obvious necessary condition for a graph to have a $k$-star decomposition is that its number of edges is divisible by $k$. Trivially, any graph has a decomposition into 1-stars. A simple inductive argument shows that any connected graph with an even number of edges has a 2-star decomposition (see \cite[Theorem 1]{CarSch}). Tarsi \cite{Tarsi1979} and Yamamoto et al. \cite{Yamamoto1975} independently proved that, for $n \geq 2$, a $k$-star decomposition of $K_n$ exists if and only if $n \geq 2k$ and $\binom{n}{2} \equiv 0 \mod{k}$. In fact, Tarsi gave necessary and sufficient conditions for the existence of a decomposition of a complete multigraph into $k$-stars while Yamamoto et al. also proved an analogous statement for complete bipartite graphs.

 A result of Dor and Tarsi \cite{DorTarsi1997} implies that determining whether an arbitrary graph $G$ has a $k$-star decomposition is $\mathsf{NP}$-complete whenever $k \geq 3$. A results of Tarsi \cite{Tarsi1981} gives a characterisation of when an arbitrary graph $G$ has a $k$-star decomposition in which the number of $k$-stars that are centred on each vertex is specified. Other results in \cite{Tarsi1981} imply various sufficient conditions for a graph to have a decomposition into $k$-stars. Hoffman and Roberts \cite{HoffmanRoberts20142}  exactly determined the maximum possible number of $k$-stars in a partial $k$-star decomposition of $K_n$ and moreover characterised the possible leaves.

This paper is concerned with the problem of when a partial $k$-star decomposition of $K_n$ can be embedded in a $k$-star decomposition of $K_{n+s}$. In 2012, Hoffman and Roberts \cite{HoffmanRoberts2014}  proved that a partial $k$-star decomposition of $K_n$ can be embedded in a $k$-star decomposition of $K_{n+s}$ for some positive integer $s$ such that $s \leq 7k-4$ when $k$ is odd and $s \leq 8k-4$ when $k$ is even. Furthermore, they conjectured that the smallest possible upper bound on $s$ is around $2k$. In 2019, Noble and Richardson \cite{NobleRichardson2019} improved the bounds on $s$ to $s \leq 3k-2$ when $k$ is odd and $s \leq 4k-2$ when $k$ is even. As our first main result of the paper we further improve these bounds.

\begin{theorem} \label{T:NRimprovement}
Let $k \geq 2$ and $n \geq 1$ be integers. Any partial $k$-star decomposition of $K_n$ can be embedded in a $k$-star decomposition of $K_{n+s}$ for some $s$ such that $s < \frac{9}{4}k$ when $k$ is odd and $s < (6-2\sqrt{2})k$ when $k$ is even.
\end{theorem}

If either of the constants $\frac{9}{4}$ or  $6-2\sqrt{2} \approx 3.17$ in the above result were decreased then the result would fail to hold for infinitely many $k$ (see Lemmas~\ref{L: even bound on s} and~\ref{L: odd bound on s}). Our next main result shows, however, that these constants can be improved if we impose a lower bound on $n$.

\begin{theorem} \label{T: statement1}
Let $k \geq 2$ and $n > \frac{k(k-1)}{\sqrt{8k}-1}$ be integers.
Any partial $k$-star decomposition of $K_n$ can be embedded in a $k$-star decomposition of $K_{n+s}$ for some $s$ such that $s \leq 2k-2$ when $k$ is odd and $s \leq 3k-2$ when $k$ is even.
\end{theorem}

Neither of the upper bounds on $s$ in this result can be decreased, no matter what lower bound we place on $n$ (see Lemmas~\ref{L:counter ex for T1 and T2}(c) and~\ref{L: tightness of T2}(b)). We prove Theorem~\ref{T: statement1} as a consequence of the following result which shows that, when $s \geq k$ and $n$ is large enough, the obvious necessary condition is also sufficient for the existence of an embedding of a partial $k$-star decomposition of $K_n$ in a $k$-star decomposition of $K_{n+s}$.

\begin{theorem} \label{T: statement2complete}
Let $k \geq 2$ and $n > \frac{k(k-1)}{\sqrt{8k}-1}$ be integers. Any nonempty partial $k$-star decomposition of $K_n$ can be embedded in a $k$-star decomposition of $K_{n+s}$ for each $s \geq k$ such that  $\binom{n+s}{2} \equiv 0 \mod{k}$.
\end{theorem}

The lower bound on $s$ in this result cannot be decreased no matter what lower bound we place on $n$ (see Lemma~\ref{L:counter ex for T1 and T2}(b)). Moreover, the lower bound on $n$ is asymptotically best possible as $k$ becomes large (see Lemma~\ref{L: bound on n}).

\section{Central functions and other preliminaries}

We introduce some more notation that we use throughout the paper.  Let $G$ be a graph. Let $E(G)$, $V(G)$ and $\overline{G}$ denote the edge set, vertex set and complement of $G$ respectively. For any $x \in V(G)$, $\deg_G(x)$ denotes the degree of $x$ in $G$. The \emph{neighbourhood} $N_G(x)$ of a vertex $x \in V(G)$  is the set of all vertices which are adjacent to $x$ in $G$. For a subset $U$ of $V(G)$ we use $G[U]$ to denote the subgraph of $G$ induced by $U$.

For a set $S$ of vertices we use $K_S$ to denote the complete graph with vertex set $S$, and for disjoint sets $S$ and $T$ of vertices we use $K_{S,T}$ to denote the complete bipartite graph with parts $S$ and $T$. For vertex-disjoint graphs $G$ and $H$ we use $G \vee H$ to denote the graph with vertex set $V(G) \cup V(H)$ and edge set $E(G) \cup E(H) \cup E(K_{V(G),V(H)})$. Our use of the notation $K_{S,T}$ will imply that $S$ and $T$ are disjoint and our use of the notation $G \vee H$ will imply that $G$ and $H$ are vertex-disjoint. As a special case, we take $G \vee K_\emptyset$ or $G \vee K_0$ to be simply the graph $G$. We can embed a partial $k$-star decomposition $\D$ of $K_n$ in a $k$-star decomposition of $K_{n+s}$ for some nonnegative integer $s$ if and only if there is a $k$-star decomposition of $L \vee K_s$, where $L$ is the leave of $\D$.

We begin by emphasising the necessary and sufficient conditions for the existence of a $k$-star decomposition of $K_n$ that we mentioned in the introduction and highlighting their effects in the special case where $k$ is a prime power.

\begin{theorem}\textup{\cite{Tarsi1979, Yamamoto1975}} \label{T: TarsiandYamamoto}
Let $k \geq 2$ and $n \geq 2$ be positive integers.
\begin{itemize}
    \item[\textup{(a)}]
A $k$-star decomposition of $K_n$ exists if and only if $n \geq 2k$ and $\binom{n}{2} \equiv 0 \mod{k}$.
    \item[\textup{(b)}]
If $k$ is a power of $2$ then a $k$-star decomposition of $K_n$ exists if and only if $n \geq 2k$ and $n \equiv 0 \mod{2k}$ or $n \equiv 1 \mod{2k}$.
    \item[\textup{(c)}]
If $k$ is a power of an odd prime then a $k$-star decomposition of $K_n$ exists if and only if $n \geq 2k$ and $n \equiv 0 \mod{k}$ or $n \equiv 1 \mod{k}$.
\end{itemize}
\end{theorem}

Parts (b) and (c) of Theorem~\ref{T: TarsiandYamamoto} follow immediately from part (a) because $\binom{n}{2} \equiv 0 \mod{k}$ is equivalent to $n \equiv 0 \mod{2k}$ or $n \equiv 1 \mod{2k}$ when $k$ is a power of 2 and is equivalent to $n \equiv 0 \mod{k}$ or $n \equiv 1 \mod{k}$ when $k$ is a power of an odd prime. We often exploit this limitation of the possible values of $n$ when $k$ is a prime power in our constructions of partial $k$-star decompositions without small embeddings.

As mentioned in the introduction, a simple inductive argument shows that any connected graph with an even number of edges has a 2-star decomposition (see \cite[Theorem 1]{CarSch}). This immediately implies the following characterisation of when a graph $L \vee K_s$ has a 2-star decomposition.

\begin{lemma}\label{L:k=2}
Let $L$ be a graph. There is a $2$-star decomposition of $L \vee K_s$ if and only if
\begin{itemize}
    \item
$s=0$ and each connected component of $L$ has an even number of edges; or
    \item
$s \geq 1$ and $|E(L \vee K_s)| \equiv 0 \mod{2}$.
\end{itemize}
\end{lemma}

Let $k \geq 2$ be an integer. In a $k$-star, the vertex of degree $k$ is called the \emph{centre}. For a given $k$-star decomposition $\D$ of $G$, we can define a function $\gamma: V(G) \rightarrow \Z^{\geq 0}$ called the \emph{central function}, where $\gamma(x)$ is the number of $k$-stars of $\D$ whose centre is $x$ for each $x \in V(G)$.  It will be helpful to bear in mind the three following properties that must hold for any central function $\gamma$ of a $k$-star decomposition of a graph $G$.
\begin{itemize}
    \item
$k\sum_{x \in V(G)}\gamma(x)=|E(G)|$.
    \item
For each edge $x_1x_2$ of $G$, $\gamma(x_1)+\gamma(x_2) \geq 1$.
    \item
For each vertex $x$ of $G$, $k\gamma(x) \leq \deg_{G}(x)$ and if $k\gamma(x) = \deg_{G}(x)$ then each edge of $G$ incident with $x$ is in a $k$-star of $\D$ centred at $x$.
\end{itemize}
We call a function $\gamma: V(G) \rightarrow \Z^{\geq 0}$ such that $k\sum_{x \in V(G)}\gamma(x)=|E(G)|$ a \emph{$k$-precentral function} for $G$. Crucial to our approach in this paper is Lemma~\ref{L:hoffman} below, which that characterises when a $k$-star decomposition of a graph $G$ with a specified central function exists. Lemma~\ref{L:hoffman} is a simple consequence of a result of Tarsi  \cite[Theorem~2]{Tarsi1981}. Because we will use Lemma~\ref{L:hoffman} so extensively, we first introduce some notation that simplifies its statement and use.

Let $\G$ be a graph $G$ equipped with a $k$-precentral function $\gamma$ (note that $G$ and $\gamma$ determine the value of $k$). We call a $k$-star decomposition of $G$ in which there are $\gamma(x)$ stars centred at $x$ for each $x \in V(G)$ a \emph{star $\G$-decomposition}. The notation we now define is implicitly dependent on $\G$, which will always be obvious from context. For any subset $T$ of $V(G)$, let $\Delta_T=\Delta^+_T-\Delta^-_T$ where $\Delta^-_T=k\sum_{x \in T}\gamma(x)$, $\Delta^+_T=|E_T|$, and $E_T$ is the set of edges of $G$ that are incident to at least one vertex in $T$. Let $\Delta$ be the minimum of $\Delta_T$ over all subsets $T$ of $V(G)$ and note that taking $T=\emptyset$ implies that $\Delta \leq 0$. Let $\T$ be the collection of subsets $T$ of $V(G)$ for which $\Delta_T=\Delta$ and which, subject to this, have minimum cardinality.
\begin{lemma}\label{L:hoffman}
Let $k \geq 2$ be an integer and let $\G$ be a graph $G$ equipped with a $k$-precentral function $\gamma$.
\begin{itemize}
    \item[\textup{(i)}]
There exists a star $\G$-decomposition if and only if $\Delta=0$.
    \item[\textup{(ii)}]
For each $T \in \T$, $T \subseteq \{x \in V(G): \gamma(x) \geq 1\}$.
\end{itemize}
\end{lemma}

\begin{proof}
We first prove (i). It is clear that a star $\G$-decomposition exists if and only if there is an orientation of the edges of $G$ such that exactly $k\gamma(x)$ edges are oriented out from $x$ for each $x \in V(G)$. Remember that $k\sum_{x \in V(G)}\gamma(x)=|E(G)|$ because $\gamma$ is a $k$-precentral function. Thus, by \cite[Theorem~2]{Tarsi1981} such an orientation exists if and only if $k\sum_{x \in S} \gamma(x) \geq |E(G[S])|$ for each subset $S$ of $V(G)$. For a given subset $S$ of $V(G)$, $k\sum_{x \in S}\gamma(x)=|E(G)|-\Delta^-_T$ and $E(G[S])=E(G) \setminus E_T$, where $T=V(G) \setminus S$. Thus, such an orientation exists if and only if
\begin{equation}\label{E:TCondition}
\Delta_T \geq 0 \qquad \text{for each subset $T$ of $V(G)$}.
\end{equation}
Because $\Delta_{\emptyset}=0$ and hence $\Delta \leq 0$, \eqref{E:TCondition} is equivalent to $\Delta=0$.

We now prove (ii). Let $T \in \T$ and suppose for a contradiction that $\gamma(x)=0$ for some $x \in T$. We have that $\Delta_{T \setminus \{x\}} \leq \Delta_T$ because $\Delta^-_{T \setminus \{x\}}=\Delta^-_T$ and $\Delta^+_{T \setminus \{x\}} \leq \Delta^+_T$ since $E_{T \setminus \{x\}} \subseteq E_T$. So, because $|T \setminus \{x\}| < |T|$, we have a contradiction to the definition of $\T$.
\end{proof}

Lemma~\ref{L:hoffman} can also be obtained by specialising results in \cite{Hoffman2004} or \cite{CameronHorsley2020} concerning star decompositions of multigraphs. Through our notation $\Delta^+_T$ and $\Delta^-_T$, the condition of Lemma~\ref{L:hoffman}(i)  is stated in the complement when compared to \cite[Theorem~2]{Tarsi1981}, but this makes it consistent with the statements in \cite{CameronHorsley2020,Hoffman2004}, which generalise more naturally to star packings of graphs.

We call a set $U$ of vertices of a graph $G$ \emph{pairwise twin}, if $N_G(x) \setminus \{y\} = N_G(y) \setminus \{x\}$ for all $x,y \in U$. The next lemma aids us when applying Lemma~\ref{L:hoffman} to graphs containing sets of pairwise twin vertices. Note that in a graph $G = L \vee K_S$, the vertices in $S$ are pairwise twin and so we can apply the lemma with $U$ chosen to be $S$.

\begin{lemma} \label{L:about S' in G}
Let $k \geq 2$ be an integer, let $G$ be a graph  and let $U$ be a pairwise twin subset of $V(G)$. Let $\G$ be the graph $G$ equipped with some $k$-precentral function $\gamma$ and let $T \in \T$. For any $x_1 \in U \setminus T$ and $x_2 \in T \cap U$ we have $\gamma(x_1)<\gamma(x_2)$. In particular, if $\gamma(x)=\gamma(x')$ for all $x, x' \in U$ then, for each $T \in \T$, either $U \subseteq T$ or $T \cap U = \emptyset$.
\end{lemma}

\begin{proof}
Suppose that $T \in \T$, $x_1 \in U \setminus T$ and $x_2 \in U \cap T$. Let $A =  N_G(x_1) \setminus T $, and note that $A=  N_G(x_2) \setminus (T \cup \{x_1\})$ because $x_1$ and $x_2$ are twin. Let $a = |A|$, $T_1= T \cup \{x_1\}$ and $T_2= T \setminus \{x_2\}$. Because $T \in \T$ and $|T_2|<|T|$ we have $\Delta_{T_1} \geq \Delta_T$ and $\Delta_{T_2} > \Delta_T$.

Observe that $\Delta^-_{T_1}=\Delta^-_{T}+k\gamma(x_1)$ and $\Delta^+_{T_1} = \Delta^+_{T}+a$ since $E_{T_1}=E_{T} \cup \{x_1z :z \in A\}$. Therefore, $\Delta_{T_1} =  \Delta_T + a -k\gamma(x_1)$ and so, because $\Delta_{T_1} \geq \Delta_T$, $k\gamma(x_1) \leq a$. Now, $\Delta^-_{T_2}=\Delta^-_{T}-k\gamma(x_2)$ and $\Delta^+_{T_2} \leq \Delta^+_{T}-a$ since $E_{T_2}=E_{T} \setminus (\{x_2z : z \in A \} \cup X)$, where $X=\{x_1\}$ if $x_1x_2 \in E(G)$ and $X=\emptyset$ if $x_1x_2 \notin E(G)$. Therefore, $\Delta_{T_2} \leq  \Delta_T - a +k\gamma(x_2)$ and so, because $\Delta_{T_2} > \Delta_T$, $a < k\gamma(x_2)$. Combining $k\gamma(x_1) \leq a$ and $a < k\gamma(x_2)$, we see we must have $\gamma(x_1)<\gamma(x_2)$.

Now suppose $\gamma(x)=\gamma(x')$ for all $x,x' \in U$. By what we have just proved, either $U \setminus T = \emptyset$ and hence $U \subseteq T$, or $T \cap U = \emptyset$.
\end{proof}

Many of the results in this paper (including Theorem~\ref{T: statement2complete}) effectively concern $k$-star decompositions of $L \vee K_s$ for some specified graph $L$ and integer $s \geq k$. Lemma~\ref{L:counter ex for T1 and T2} below illustrates why we usually impose the condition that $s$ be at least $k$ in these results. First we state a special case of a result of Tarsi \cite[Theorem 4]{Tarsi1981} that we will often use to show that a certain graph is the leave of a partial $k$-star decomposition.

\begin{theorem}[\cite{Tarsi1981}] \label{T: Tarsi's result}
Let $G$ be a graph of order $n$ such that $\deg_G(x) \geq \frac{1}{2}n+k-1$ for each $x \in V(G)$. Then $G$ has a $k$-star decomposition if $|E(G)| \equiv 0 \mod{k}$.
\end{theorem}

\begin{lemma} \label{L:counter ex for T1 and T2}
Let $k \geq 2$ and $n \geq 2$ be integers such that $k$ is odd and $n \equiv 2 \mod{2k}$. Let $L$ be a graph of order $n$ that has exactly one edge.
\begin{itemize}
    \item[\textup{(a)}]
There is a partial $k$-star decomposition of $K_n$ whose leave is $L$.
    \item[\textup{(b)}]
There is no $k$-star decomposition $L \vee K_{k-1}$, even though $|E(L \vee K_{k-1})| \equiv 0 \mod{k}$.
    \item[\textup{(c)}]
If $k$ is a power of an odd prime, there is no $k$-star decomposition $L \vee K_{s}$ for any $s<2k-2$.
\end{itemize}

\end{lemma}

\begin{proof}
We first prove (a) by showing that a $k$-star decomposition of $\overline{L}$ exists. This is trivial if $n=2$. If $n  \geq 2k+2$, then $\deg_{\overline{L}}(y) \geq n-2 \geq \frac{1}{2}n+k-1$ for each $y \in V(L)$ and $|E(\overline{L})|=\binom{n}{2}-1 \equiv 0 \mod{k}$  since $n \equiv 2 \mod{2k}$. Therefore, by Theorem~\ref{T: Tarsi's result}, a $k$-star decomposition of $\overline{L}$ exists.

We now prove (b). Note that $|E(L \vee K_{k-1})|=1+n(k-1)+\binom{k-1}{2}  \equiv 0 \mod{k}$ because $n \equiv 2 \mod{2k}$ and $k$ is odd. Let $r$ be the nonnegative integer such that $n= 2kr+2$. Suppose for a contradiction that there is a $k$-star decomposition $\D$ of $L \vee K_S$, where $|S|=k-1$, and let $\gamma$ be the central function of $\D$. Now $|E(L \vee K_S)|=1+n(k-1)+\binom{k-1}{2}$ and so $\sum_{x \in V(L) \cup S}\gamma(x) = (2r+\frac{1}{2})(k-1)+1$. Observe that $\deg_{L \vee K_S}(y_1)=\deg_{L \vee K_S}(y_2)=k$, where $y_1y_2$ is the only edge in $L$, and $\deg_{L \vee K_S}(y)=k-1$ for each $y \in V(L)\setminus \{y_1,y_2\}$. So, without loss of generality, $\gamma(y_1)=1$, every edge of $L \vee K_S$ incident with $y_1$ is in the star in $\D$ centred at $y_1$, and $\gamma(y)=0$ for each $y \in V(L)\setminus \{y_1\}$. Thus $\sum_{z \in S}\gamma(z)=(2r+\frac{1}{2})(k-1)$. By the pigeonhole principle, it follows that $\gamma(z_1)=2r+1$ for some $z_1 \in S$ because $|S|=k-1$. Now $\deg_{L \vee K_S}(z_1) = n +k-2 = k(2r + 1)$ noting that $n = 2kr +2$. So every edge incident with $z_1$ is in a star in $\mathcal{D}$ centred at $z_1$. But this contradicts the fact that the edge $y_1z_1$ is in the star in $\D$ centred at $y_1$.

We now prove (c). Suppose that $k$ is a power of an odd prime. Assume for a contradiction that $\D$ is a $k$-star decomposition of $L \vee K_S$ where $|S|=s$ for some nonnegative integer $s < 2k-2$. By part (a) of this lemma and Theorem~\ref{T: TarsiandYamamoto}(c),  we have that $n+s \equiv 0 \mod{k}$ or $n+s \equiv 1 \mod{k}$ and hence, because $n \equiv 2 \mod{2k}$, that $s \equiv k-2 \mod{k}$ or $s \equiv k-1 \mod{k}$. So $s \in \{k-2,k-1\}$ because $s < 2k-2$. So then $s = k-1$ because a $k$-star in $\D$ must be centred at an end vertex of the edge in $L$ and these vertices have degree $s+1$ in $L \vee K_S$. However, a $k$-star decomposition of $L \vee K_{k-1}$ does not exist by (b).
\end{proof}

\section{Embedding maximal partial \texorpdfstring{$\bm{k}$}{k}-star decompositions}

A partial $k$-star decomposition of a graph $G$ is \emph{maximal} if there is no star that can be added to it to produce a partial $k$-star decomposition of $G$ containing more stars. Thus, a partial $k$-star decomposition of a graph $G$ is maximal if and only if its leave has maximum degree at most $k-1$. In this section we prove results about embedding maximal partial $k$-star decompositions of $K_n$ in $k$-star decompositions of $K_{n+s}$ where $s \geq k$. These results will be crucial in proving the main theorems.

An \emph{independent set} in a graph is a set of its vertices that are pairwise non-adjacent. The \emph{independence number} $\alpha(G)$ of a graph $G$ is the maximum cardinality of an independent set in $G$. In \cite[Corollary 2]{CaroRoditty1985}, Caro and Roditty note that if a graph $G$ has a decomposition into $k$-stars then $\alpha(G) \geq |V(G)|-\frac{1}{k}|E(G)|$. This can be seen by observing that any edge in $G$ must have a star of the decomposition centred on at least one of its end-vertices. For the cases we are interested in, we formalise this observation in the following lemma.

\begin{lemma}\label{l:obstacle}
Let $k \geq 2$, $n \geq 1$ and $s \geq 0$ be integers, and let $L$ be a graph of order $n$. If there is $k$-star decomposition of $L \vee K_s$, then $\alpha(L) \geq n+s-\frac{1}{k}|E(L \vee K_s)|$.
\end{lemma}

\begin{proof}
If there is a $k$-star decomposition of $L \vee K_s$, then $\alpha(L \vee K_s) \geq n+s-\frac{1}{k}|E(L \vee K_s)|$ by \cite[Corollary 2]{CaroRoditty1985}. Furthermore, it is easy to see that $\alpha(L \vee K_s)=\alpha(L)$.
\end{proof}

In this section we show that, for a maximal partial $k$-star decomposition $\D$ of $K_n$ and an integer $s \geq k$ such that $\binom{n+s}{2}\equiv 0 \mod{k}$, the obstacle described by Lemma~\ref{l:obstacle} is the only thing that can prevent the existence of an embedding of $\D$ in a $k$-star decomposition of $K_{n+s}$. We do this in two lemmas: Lemma~\ref{L: no of total edges are small} deals with the case where the number of stars to be added is small and the obstacle may arise whereas Lemma~\ref{L: no of total edges are large} deals with the case where the number of stars to be added is large and the obstacle cannot arise.

\begin{lemma}\label{L: no of total edges are small}
Let $k$, $n$ and $s$ be integers with $s \geq k \geq 2$, and let $L$ be a graph of order $n$ with maximum degree at most $k-1$ and $|E(L \vee K_s)| \leq k(n+s)$. Then there is a $k$-star decomposition of $L \vee K_s$ if and only if $|E(L \vee K_s)| \equiv 0 \mod{k}$ and $\alpha(L) \geq n+s-\frac{1}{k}|E(L \vee K_s)|$.
\end{lemma}

\begin{proof}
The `only if' direction follows from Lemma~\ref{l:obstacle}, so we only need to prove the `if' direction.

Suppose that $|E(L \vee K_{S})| \equiv 0 \mod{k}$, where $S$ is a set with $|S|=s$. Let $V=V(L \vee K_S)$ and $b=\frac{1}{k}|E(L \vee K_S)|$, and suppose that $L$ has an independent set $A$ containing $n+s-b$ vertices. Note that $n+s-b \geq 0$ because $|E(L \vee K_{S})| \leq k(n+s)$ by our hypotheses. Define a  $k$-precentral function $\gamma$ for $L \vee K_{S}$ by $\gamma(x) = 0$ for each $x \in A$ and $\gamma(x) =1$ for each $x \in V \setminus A$. This is indeed a $k$-precentral function for $L \vee K_{S}$ because $\sum_{x \in V}\gamma(x)=n+s-|A|=b$. Let $\G$ be the graph $L \vee K_{S}$ equipped with $\gamma$. We complete the proof by showing that $\Delta = 0$ and hence a star $\G$-decomposition exists by Lemma~\ref{L:hoffman}. Let $T \in \T$ and suppose for a contradiction that $\Delta_T < 0$. Since $\gamma(z)=1$ for all $z \in S$, we can apply Lemma~\ref{L:about S' in G} with $U=S$ to conclude that either $T \cap S= \emptyset$ or $S \subseteq T$. We consider these cases separately, with the latter splitting into two subcases.

\textbf{Case 1:} Suppose that $T \cap S= \emptyset$. This implies $T \subseteq V(L)$. Then $\Delta^+_T \geq s|T|$, because $E(K_{S,T}) \subseteq E_T$ and $\Delta^-_T = k|T|$ by the definition of $\gamma$ and Lemma~\ref{L:hoffman}(ii). Therefore, we have $\Delta^-_T \leq \Delta^+_T $ as $s \geq k$. This contradicts $\Delta_T <0$.

\textbf{Case 2a:} Suppose that $S \subseteq T$ but $T \neq V \setminus A$. Then there is a vertex $y \in V(L) \setminus (A \cup T)$ and, by the definition of $\gamma$, $\gamma(y)=1$. Let $T_1 = T \cup \{y\}$. Then $\Delta^+_{T_1} \leq \Delta^+_{T}+k-1$, noting that $\deg_L(y) \leq k-1$ and $\Delta^-_{T_1} = \Delta^-_{T}+k$. Therefore, $\Delta_{T_1} \leq \Delta_{T}-1$ contradicting $T \in \T$.

\textbf{Case 2b:} Suppose that $T=V \setminus A$. Then $\Delta^+_T = |E(L \vee K_S)|$ because $E_T=E(L \vee K_S)$ since $A$ is independent. Moreover, $\Delta^-_T = |E(L \vee K_S)|$ because $\gamma$ is a $k$-precentral function for $L \vee K_S$. So $\Delta^+_T = \Delta^-_T$ contradicting $\Delta_T<0$.
\end{proof}

Note that the condition $n \geq k$ in the following lemma will certainly hold whenever $L$ is the leave of a nontrivial $k$-star decomposition.

\begin{lemma} \label{L: no of total edges are large}
Let $k$, $n$ and $s$ be positive integers with $s \geq k \geq 2$ and $n \geq k$, and let $L$ be a graph of order $n$ with maximum degree at most $k-1$ and $|E(L \vee K_s)| \geq k(n+s)$. Then there is a $k$-star decomposition of $L \vee K_s$ if and only if $|E(L \vee K_s)| \equiv 0 \mod{k}$.

\end{lemma}

\begin{proof}
If $L \vee K_s$ has a $k$-star decomposition, then obviously $|E(L \vee K_s)| \equiv 0 \mod{k}$. So it suffices to prove the `if' direction.

Assume that $|E(L \vee K_{S})| \equiv 0 \mod{k}$, where $S$ is a set with $|S|=s$, let $b=\frac{1}{k}|E(L \vee K_{S})|$ and note $b \geq n+s$ by the hypotheses of the lemma. Thus, we can define a $k$-precentral function $\gamma$ on $L \vee K_{S}$ such that $\gamma(y)=1$ for each $y \in V(L)$ and $\gamma(z)\in\{d,d+1\}$ for each $z \in S$, where $d=\lfloor\frac{b-n}{s}\rfloor$. Note that $d \geq 1$ since $b \geq n+s$ and let $S_0=\{z \in S: \gamma(z)=d\}$. We will show there is a star $\G$-decomposition where $\G$ is $L \vee K_S$ equipped with $\gamma$.

Let $T \in \T$, $H=L[V(L) \setminus T]$, $h=|V(H)|$, and $e=|E(H)|$. By Lemma~\ref{L:hoffman}, it suffices to show that $\Delta_T \geq 0$. By Lemma~\ref{L:about S' in G} with $U=S$, we have that $T \cap S \in \{\emptyset,S \setminus S_0,S\}$. We separate the proof into three cases accordingly.

\smallskip
\noindent\textbf{Case 1:} Suppose that $T \cap S = \emptyset$. Then $T = V(L) \setminus V(H)$. Noting that $E_T=E(K_{S,V(L) \setminus V(H)}) \cup (E(L) \setminus E(H))$ and $\Delta_T^-=k(n-h)$, we have
\[\Delta_T = \bigl((n-h)s+|E(L)|-e\bigr)-k(n-h) = (n-h)(s-k)+|E(L)|-e.\]
This last expression is nonnegative because $n \geq h, s \geq k$ and $|E(L)| \geq e$.

\smallskip
\noindent\textbf{Case 2:} Suppose that $T \cap S = S$. Noting that $E_T=E(L \vee K_{S}) \setminus E(H)$, that $|E(L \vee K_{S})|=bk$, and that $\Delta_T^-=k(b-h)$, we see that
\[\Delta_T = (bk-e)-k(b-h) = kh-e.\]
This last expression is nonnegative because $e \leq \frac{1}{2}h(k-1)$ since $H$ has maximum degree at most $k-1$.

\smallskip
\noindent\textbf{Case 3:} Suppose that $T \cap S = S \setminus S_0$. Let $s_0=|S_0|$. Noting that
\[E_T=E(L \vee K_{S}) \mathbin{\big\backslash} \bigl(E(K_{S_0}) \cup E(K_{S_0,V(H)}) \cup E(H)\bigr),\] that $|E(L \vee K_{S})|=bk$, and that $\Delta_T^-=k(b-ds_0-h)$, we see that
\begin{equation}\label{e:deltaEquality}
\Delta_T=\Bigl(bk-\mbinom{s_0}{2}-hs_0-e\Bigr)-k(b-ds_0-h)=\mfrac{s_0}{2}(2dk+1-s_0)+h(k-s_0)-e.
\end{equation}
The remainder of the proof is a somewhat tedious verification that this last expression is nonnegative. We first observe the following three useful facts.
\begin{itemize}
    \item[(F1)]
$2e \leq h(k-1)$
    \item[(F2)]
$e \leq k (n + s (d+1) - s_0) -n s - \binom{s}{2}$
    \item[(F3)]
$d \leq \frac{1}{2ks}(n(2s-k-1) + s(s-1) + 2 k s_0)-1$
\end{itemize}
Note that (F1) holds because $H$ is a subgraph of $L$ and thus has maximum degree at most $k-1$. Also, (F2) holds because $e \leq |E(L)| = bk-ns-\binom{s}{2}$ and $b=n+s(d+1)-s_0$ from the definition of $\gamma$. Further, (F3) holds because $b=n+s(d+1)-s_0$, $b=\frac{1}{k}(|E(L)|+ns+\binom{s}{2})$ and $|E(L)| \leq \frac{1}{2}n(k-1)$ since $L$ has maximum degree at most $k-1$.  We divide this case into subcases depending on the value of $s_0$.

\smallskip
\noindent\textbf{Case 3a:} Suppose that $s_0 \geq k$. Then substituting $h \leq n$ and  (F2) into \eqref{e:deltaEquality} we obtain
\begin{equation}\label{e:deltaBound1}
\Delta_T \geq \mfrac{s-s_0}{2}\bigl(s+s_0+2(n-k)-2dk-1\bigr).
\end{equation}
Substituting (F3) into \eqref{e:deltaBound1} and rearranging, we obtain
\begin{equation*}
\Delta_T \geq \mfrac{s-s_0}{2s}\bigl((s_0-k)(s-k)+k(s+n-s_0-k)+n\bigr).
\end{equation*}
This last expression is nonnegative because $n \geq k$ and $s \geq s_0 \geq k$ using the conditions of this case.

\smallskip
\noindent\textbf{Case 3b:} Suppose that $s_0 \leq \frac{k+1}{2}$. Then substituting $e \leq \frac{1}{2}h(k-1)$ from (F1) into \eqref{e:deltaEquality} we obtain
\begin{equation*}
\Delta_T \geq \mfrac{s_0}{2}(2dk+1-s_0)+h(\tfrac{k+1}{2}-s_0).
\end{equation*}
This last expression can be seen to be nonnegative using  $d \geq 1$ and $1 \leq s_0 \leq \frac{k+1}{2}$  from the conditions of this case.

\smallskip
\noindent\textbf{Case 3c:} Suppose that $\frac{k+2}{2} \leq s_0 \leq k-1$. Then substituting $h \geq \frac{2e}{k-1}$ from (F1) into \eqref{e:deltaEquality} we obtain
\begin{equation}\label{e:deltaBound3i}
\Delta_T \geq \mfrac{s_0}{2}(2dk+1-s_0)-\mfrac{e}{k-1}(2s_0-k-1).
\end{equation}
Observing that $2s_0-k-1 > 0$ by the conditions of this case, substituting  (F2) and rearranging, we obtain
\begin{equation}\label{e:deltaBound3ii}
\Delta_T \geq \mfrac{2s_0-k-1}{k-1}\left(\mbinom{s}{2} + n(s-k) - k(s - s_0 )\right)-\mbinom{s_0}{2}+\mfrac{dk}{k-1}\bigl(s (k + 1) - s_0 (2 s - k + 1)\bigr).
\end{equation}
We further divide this subcase according to the sign of the coefficient of $d$ in (\ref{e:deltaBound3ii}).

\smallskip
\noindent\textbf{Case 3c(i):} Suppose that $s (k + 1) < s_0 (2 s - k + 1)$. Substituting  (F3) into (\ref{e:deltaBound3ii}) and simplifying, we obtain
\begin{equation}\label{e:deltaBound3bv}
\Delta_T \geq \mfrac{s-s_0}{2s}\bigl(n+k(n-s_0)+s_0(s-k)\bigr).
\end{equation}
We can easily see that $\Delta_T$ is nonnegative since  $s \geq k$, $n \geq k$ and $s_0 \leq k-1$ by the conditions of Case 3c.

\smallskip
\noindent\textbf{Case 3c(ii):} Suppose that $s (k + 1) \geq s_0 (2 s - k + 1)$.
Substituting $d \geq 1$ and $n \geq k$ in (\ref{e:deltaBound3ii}) and rearranging yields
\begin{equation}\label{e:deltaBound3aiii}
\Delta_T \geq  \mfrac{2s_0-k-1}{2(k-1)}\left(s^2-(2k+1)s-2k^2\right)+\mfrac{3k+1}{k-1}\mbinom{s_0}{2}.
\end{equation}
Recall that $2s_0 > k-1$ by the conditions of Case 3c. Since $s \geq k$ is an integer, either $s = k$ or $s \geq k+1$, and hence $s^2-(2k+1)s \geq -k(k+1)$. Substituting this into \eqref{e:deltaBound3aiii} and rearranging, we obtain

\begin{equation}\label{e:deltaBound3aiv}
\Delta_T \geq \mfrac{3k+1}{k-1}\mbinom{k-s_0+1}{2}.
\end{equation}
This last expression is clearly nonnegative since $s_0 \leq k-1$ by the conditions of Case 3c.
\end{proof}

\section{Proof of Theorems~~\ref{T: statement1} and \ref{T: statement2complete}}\label{S:proofOfTh1Th2}

Caro \cite{Caro} and Wei \cite{Wei1981} independently established the following lower bounds on the independence number of a graph.

\begin{theorem}[\cite{Caro}, \cite{Wei1981}]\label{T:CaroWei}
For any graph $G$, the following hold.
\begin{itemize}
    \item[\textup{(a)}]
$\alpha(G) \geq \sum_{x \in V(G)} \frac{1}{\deg_G(x) + 1}$
    \item[\textup{(b)}]
$\alpha(G) \geq \frac{|V(G)|^2}{2|E(G)| + |V(G)|}$
\end{itemize}
\end{theorem}\medskip

Part (b) of Theorem~\ref{T:CaroWei} follows immediately from part (a) because, by convexity,
\[\medop\sum_{x \in V(G)} \mfrac{1}{\deg_G(x) + 1} \geq \mfrac{|V(G)|}{d+1} \qquad \text{where} \qquad d=\mfrac{2|E(G)|}{|V(G)|}.\]
In Lemma~\ref{L: newconditions on s} below we combine Theorem~\ref{T:CaroWei}(b) with Lemmas~\ref{L: no of total edges are small} and~\ref{L: no of total edges are large} to show that, for any graph $L$, a $k$-star decomposition of $L \vee K_s$ must exist if $|E(L \vee K_s)| \equiv 0 \mod{k}$ and $s$ is greater than a certain function of $k$ and $|V(L)|$. Theorem~\ref{T: statement2complete} then follows from Lemma~\ref{L: newconditions on s} and, in turn, Theorem~\ref{T: statement1} follows from Theorem~\ref{T: statement2complete}. For technical reasons we restrict Lemma~\ref{L: newconditions on s} to $k \geq 3$. Lemma~\ref{L:k=2} covers the case when $k=2$.

\begin{lemma} \label{L: newconditions on s}
Let $k$, $n$ and $s$ be positive integers with $s \geq k \geq 3$ and $n \geq k$, and let $L$ be a graph of order $n$ such that $|E(L \vee K_s)| \equiv 0 \mod{k}$. Then there is a $k$-star decomposition of $L \vee K_s$ if
\begin{equation}\label{e:sLowerBound}
s > k-n+\mfrac{1}{2} + \sqrt{\left(n-\sqrt{2k}\right)^2+k(k-3)+\tfrac{1}{4}}.
\end{equation}
In particular, such a decomposition exists if  $n > \frac{k(k-1)}{\sqrt{8k}-1}$.
\end{lemma}

\begin{proof}
Observe that the right hand side of \eqref{e:sLowerBound} is real because $k \geq 3$. We first prove the first part of the lemma. Suppose that  \eqref{e:sLowerBound} holds. We may assume that $L$ has maximum degree at most $k-1$ because otherwise we can greedily delete $k$-stars from $L$ until this is the case, apply the proof, and finally add the deleted $k$-stars to the decomposition produced. Let $b = \frac{1}{k}|E(L \vee K_s)|$, note that $b$ is an integer because $|E(L \vee K_s)| \equiv 0 \mod{k}$, and let  $e = |E(L)|$. If $b\geq n+s$, then a $k$-star decomposition of $L \vee K_s$ exists by Lemma~\ref{L: no of total edges are large}, so we may assume that $b<n+s$. By Lemma~\ref{L: no of total edges are small} it suffices to show that $\alpha(L) \geq n+s-b$.

By Theorem~\ref{T:CaroWei} we have $\alpha(L) \geq \frac{n^2}{2e+n}$. So, because $\alpha(L)$ and $n+s-b$ are both integers, it is enough to show that $\frac{n^2}{2e+n} > n+s-b-1$.  Using $b=\frac{1}{k}(e+ns+\binom{s}{2})$ and multiplying through by $2k$, this is equivalent to showing that
\begin{equation} \label{E:testForS}
s^2+(2n-2k-1)s-2kn+2k+2e+\mfrac{2kn^2}{(2e+n)}
\end{equation}
is positive.
Considered as a function of a real variable $e \geq 0$, \eqref{E:testForS} is minimised when $e=\frac{n}{2}(\sqrt{2k}-1)$. Substituting this value for $e$ and rearranging, we see that \eqref{E:testForS} is at least
\[s^2 + (2n-2k-1)s + 2k  - \left(2k-2 \sqrt{2k}+1\right)n.\]
Considering this last expression as a quadratic in $s$, it can be seen that it is positive when  \eqref{e:sLowerBound} holds. Thus, \eqref{E:testForS} is positive and $\alpha(L) \geq n+s-b$, as required.

We now prove the second part of the lemma. Suppose that $n > \frac{k(k-1)}{\sqrt{8k}-1}$.  Since $s \geq k$, substituting $s=k$ into \eqref{e:sLowerBound} and rearranging shows that \eqref{e:sLowerBound} will hold if
\[n-\mfrac{1}{2} > \sqrt{\left(n-\sqrt{2k}\right)^2+k(k-3)+\tfrac{1}{4}}.\]
By squaring both sides of this expression and rearranging, we see that it is equivalent to  $n > \frac{k(k-1)}{\sqrt{8k}-1}$.
Therefore, by the first part of the lemma, a $k$-star decomposition of $L \vee K_s$ exists.
\end{proof}

We can now prove Theorem~\ref{T: statement2complete} directly from Lemma~\ref{L: newconditions on s}.

\begin{proof} [\textbf{\textup{Proof of Theorem~\ref{T: statement2complete}}}]
Let $L$ be the leave of a nonempty partial $k$-star decomposition of $K_n$ and note that this implies that $n > k$. Let $s$ be an integer such that $s \geq k$ and $\binom{n+s}{2} \equiv 0 \mod{k}$. Since $L$ is the leave of a partial $k$-star decomposition and $\binom{n+s}{2} \equiv 0 \mod{k}$, it follows that $|E(L \vee K_s)| \equiv 0 \mod {k}$. So, by Lemma~\ref{L: newconditions on s} if $k \geq 3$ and by Lemma~\ref{L:k=2} if $k=2$, there is a $k$-star decomposition of $L \vee K_s$.
\end{proof}

Lemma~\ref{L:counter ex for T1 and T2}(b) demonstrates that the lower bound on $s$ in Theorem~\ref{T: statement2complete} cannot be decreased no matter what lower bound we place on $n$. Next, in Lemma~\ref{L: bound on n}, we show that in the case $s=k$ the lower bound on $n$ in Theorem~\ref{T: statement2complete} is asymptotically best possible. To see that Lemma~\ref{L: bound on n} implies this, note that $\frac{k(k-1)}{\sqrt{8k}-1}=(\frac{k}{2})^{3/2}+O(k)$ as $k$ becomes large.

\begin{lemma} \label{L: bound on n}
Let $k =2^t$ for some odd integer $t \geq 7$, let $m = \sqrt{2k}$, and let $n = \frac{1}{4}km-k = (\frac{k}{2})^{3/2}-k$. Let $L$ be a graph of order $n$ that is a vertex disjoint union of $\frac{n}{m}$ copies of $K_m$. Then a partial $k$-star decomposition of $K_n$ whose leave is $L$ exists and furthermore it cannot be embedded in a $k$-star decomposition of $K_{n+k}$, even though $\binom{n+k}{2} \equiv 0 \mod{k}$.
\end{lemma}

\begin{proof}
Note that $m=2^{(t+1)/2}$ is an integer divisible by 8 because $t$ is odd and $t \geq 7$. Thus $n \equiv k \mod{2k}$, $\frac{n}{m}$ is an integer and $\binom{n+k}{2} \equiv 0 \mod{k}$. Note that $|E(L)|=\frac{n}{m}\binom{m}{2}=\frac{n}{2}(m-1)$. We first show that $L$ is the leave of a partial $k$-star decomposition of $K_n$. Note that $\deg_{\overline{L}}(y) = n-m \geq \frac{1}{2}n+k-1$ for each $y \in V(L)$ because $n= \frac{1}{4}km-k$ and $k \geq 128$. Furthermore, $E(\overline{L})=\binom{n}{2}-\frac{n}{2}(m-1) = \frac{n}{2}(n-m) \equiv 0 \mod{k}$ because $n \equiv 0 \mod{k}$ and $n-m$ is even. Therefore, by Theorem~\ref{T: Tarsi's result}, there is a $k$-star decomposition of $\overline{L}$.

We complete the proof by using Lemma~\ref{l:obstacle} to show that there is no $k$-star decomposition of $L \vee K_k$. Observe that
\[n+k-\mfrac{1}{k}|E(L \vee K_k)|=n+k-\mfrac{1}{k}\left(\mfrac{n}{2}(m-1)+kn+\mbinom{k}{2}\right)= \mfrac{k}{4}+\mfrac{5\sqrt{2k}}{8}\]
where the first equality follows using $|E(L)|=\frac{n}{2}(m-1)$ and the second follows using $n=\frac{1}{4}km-k$ and $m=\sqrt{2k}$. On the other hand, $\alpha(L)=\frac{n}{m}=\frac{k}{4}-\frac{k}{m}$ because an independent set in $L$ can contain at most one vertex from each copy of $K_m$. So we have $\alpha(L)<n+k-\frac{1}{k}|E(L \vee K_k)|$ and hence there is no $k$-star decomposition of $L \vee K_k$ by Lemma~\ref{l:obstacle}.
\end{proof}

Theorem~\ref{T: statement1} follows readily from Theorem~\ref{T: statement2complete}.

\begin{proof}[\textbf{\textup{Proof of Theorem~\ref{T: statement1}}}]
Let $\D$ be a partial $k$-star decomposition of $K_n$. If $\D$ is empty and $n$=1, then $\D$ is trivially its own embedding. If $\D$ is empty and $n \geq 2$, then there is an embedding of $\D$ in a $k$-star decomposition of $K_{2k}$ by Theorem~\ref{T: TarsiandYamamoto}(a). So in either case the result holds, and hence we may assume that $\D$ is nonempty.

If $k$ is even, let $s$ be an element of $\{k,\ldots,3k-2\}$ such that $n+s \equiv 0 \mod{2k}$ or $n+s \equiv 1 \mod{2k}$. If $k$ is odd, let $s$ be an element of $\{k,\ldots,2k-2\}$ such that $n+s \equiv 0 \mod{k}$ or $n+s \equiv 1 \mod{k}$. In either case such an $s$ exists because $\{k,\ldots,3k-2\}$ contains $2k-1$ consecutive integers and $\{k,\ldots,2k-2\}$ contains $k-1$ consecutive integers. Then $\binom{n+s}{2} \equiv 0 \mod{k}$ by our definition of $s$. So by Theorem~\ref{T: statement2complete} there is an embedding of $\D$ in a $k$-star decomposition of $K_{n+s}$ and hence the result is proved.
\end{proof}

Lemma~\ref{L:counter ex for T1 and T2}(c) shows that the upper bound of $2k-2$ on $s$ in the $k$ odd case of Theorem~\ref{T: statement1} cannot be improved for any $k$ that is a power of an odd prime. Next, in Lemma~\ref{L: tightness of T2}, we show that the upper bound of $3k-2$ on $s$ in the $k$ even case of Theorem~\ref{T: statement1} cannot be improved for any $k \geq 16$ that is a power of 4.

\begin{lemma} \label{L: tightness of T2}
Let $k = 2^t$ for some even $t \geq 4$, and let $n \geq 3k+2$ be an integer such that $n \equiv k+2 \mod{2k}$. Let $L$ be a graph of order $n$ that is a vertex disjoint union of one copy of $K_{\sqrt{k}}$, $\frac{1}{2}\sqrt{k}+1$ copies of $K_2$ and $n-2\sqrt{k}-2$ copies of $K_1$. A partial $k$-star decomposition of $K_n$ whose leave is $L$ exists and furthermore it cannot be embedded in a $k$-star decomposition of $K_{n+s}$ for any $s < 3k-2$.

\end{lemma}

\begin{proof}
A simple calculation shows that $|E(L)|=\frac{1}{2}(k+2)$.
We first show that  $L$ is the leave of a partial $k$-star decomposition of $K_n$. Note that $\deg_{\overline{L}}(y) \geq n-\sqrt{k} \geq \frac{1}{2}n+k-1$ for each $y \in V(L)$ since $n \geq 3k+2$ and $k \geq 16$. Furthermore, $|E(\overline{L})| = \binom{n}{2}-\frac{1}{2}(k+2) \equiv 0 \mod{k}$ since $n \equiv k+2 \mod{2k}$. Therefore, a $k$-star decomposition of $\overline{L}$ exists by Theorem~\ref{T: Tarsi's result}.

Now assume for a contradiction that $\D$ is a $k$-star decomposition of $L \vee K_S$ where $|S|=s$ for some nonnegative integer $s < 3k-2$ and let $\gamma$ be the central function of $\D$. We must have that $n+s \equiv 0 \mod{2k}$ or $n+s \equiv 1 \mod{2k}$ by Theorem~\ref{T: TarsiandYamamoto}(b) and hence, because $n \equiv k+2 \mod{2k}$, that $s \equiv k-2 \mod{2k}$ or $s \equiv k-1 \mod{2k}$. Therefore, $s \in \{k-2,k-1\}$ since $s < 3k-2$.

Let $V_1$ be the vertex set of the copy of $K_{\sqrt{k}}$ in $L$ and let $V_2$ be the set of vertices in the $\frac{1}{2}\sqrt{k}+1$ copies of $K_2$ in $L$. If $s=k-2$, then $\deg_{L \vee K_S}(y)=k-1$ and hence $\gamma(y)=0$ for each $y \in V_2$ which contradicts the fact that each edge in $L[V_2]$ is in a star in $\mathcal{D}$. Thus it must be that $s=k-1$ and $\D$ is a $k$-star decomposition of $L \vee K_{k-1}$. Let $r$ be the positive integer such that $n= 2kr+k+2$.  Observe the following.
\begin{itemize}
    \item
$\sum_{x \in V(L) \cup S}\gamma(x) =  (2r+1)(k-1)+\frac{1}{2}k+1$ because $|E(L \vee K_{k-1})| = \frac{1}{2}(k+2)+n(k-1)+\binom{k-1}{2}$.
    \item
$\sum_{y \in V_1}\gamma(y) \leq \sqrt{k}$ because $\deg_{L \vee K_{S}}(y)=k+\sqrt{k}-2<2k$ for each $y \in V_1$ and hence $\gamma(y) \leq 1$ for all $y \in V_1$.
    \item
$\sum_{y \in V_2}\gamma(y) = \frac{1}{2}\sqrt{k}+1$ because $\deg_{L \vee K_S}(y)=k$ for each $y \in V_2$ and hence $\gamma(y_1)+\gamma(y_2)=1$ for each edge $y_1y_2$ in $L[V_2]$.
    \item
$\sum_{y \in V(L) \setminus (V_1 \cup V_2)}\gamma(y) = 0$ because $\deg_{L \vee K_S}(y) = k-1$ for each $y \in V(L) \setminus (V_1 \cup V_2)$.
\end{itemize}
Using these four facts and simplifying we have
\[\medop\sum_{z \in S  }\gamma(z) = \medop\sum_{x \in V(L) \cup S}\gamma(x) - \medop\sum_{y \in V(L)  }\gamma(y) \geq  (2r+1)(k-1)+\tfrac{1}{2}k -\tfrac{3}{2}\sqrt{k}>(2r+1)(k-1)\]
where the last inequality follows because $k \geq 16$. So, by the pigeonhole principle, $\gamma(z_1) \geq 2r+2$ for some $z_1 \in S$ because $s=k-1$. Now $\deg_{L \vee K_S}(z_1) = n+k-2 = k(2r + 2)$ noting that $n = 2kr +k+2$, so it must be that $\gamma(z_1) = 2r+2$ and that every edge incident with $z_1$ is in a star in $\mathcal{D}$ centred at $z_1$. But this contradicts the fact that, for any vertex $y_1 \in V_2$ such that $\gamma(y_1)=1$, the edge $y_1z_1$ must be in a star in $\mathcal{D}$ centred at $y_1$.
\end{proof}

\section{Proof of Theorem~\ref{T:NRimprovement}}\label{S:proofOfTh3}

From Lemma~\ref{L: newconditions on s}, it is not too difficult to prove Theorem~\ref{T:NRimprovement} in the case where $k$ is even. Note that in fact the argument in the proof also applies when $k$ is odd.

\begin{lemma}\label{L:NRimprovementEven}
Let $k \geq 2$ and $n \geq 1$ be integers. Any partial $k$-star decomposition of $K_n$ can be embedded in a $k$-star decomposition of $K_{n+s}$ for some $s$ such that $s < (6-2\sqrt{2})k$.
\end{lemma}

\begin{proof}
Let $\D$ be a partial $k$-star decomposition of $K_n$ and $L$ be its leave. Note that we will have $|E(L \vee K_s)| \equiv 0 \mod{k}$ for any integer $s$ such that $n+s \equiv 0 \mod{2k}$. If $k=2$ then we can choose $s \in \{1,2,3,4\}$ such that $n+s \equiv 0 \mod{4}$ and $L \vee K_s$ will have a $2$-star decomposition by Lemma~\ref{L:k=2}, so we may assume $k \geq 3$. We consider three cases according to the value of $n$.

\smallskip
\noindent\textbf{Case 1:} Suppose that $n \geq 2\sqrt{2}k$. Let $s$ be an integer such that $(4-2\sqrt{2})k \leq s < (6-2\sqrt{2})k$ and $n+s \equiv 0 \mod{2k}$. By Lemma~\ref{L: newconditions on s} there is a $k$-star decomposition of $L \vee K_s$ and hence the result is proved provided that \eqref{e:sLowerBound} holds. The lower bound on $s$ given by \eqref{e:sLowerBound} can be seen to be decreasing in $n$, so it suffices to show that this bound is less than $(4-2\sqrt{2})k$ when $n = 2\sqrt{2}k$. Substituting $n = 2\sqrt{2}k$ into the bound gives
\[\left(1-2\sqrt{2}\right)k+\tfrac{1}{2}+\sqrt{9k^2-8k^{3/2}-k+\tfrac{1}{4}}\]
which is easily seen to be less than $(4-2\sqrt{2})k$ since the final term is less than $3k -\frac{1}{2}$.

\smallskip
\noindent\textbf{Case 2:} Suppose that $k+1 \leq n < 2\sqrt{2}k$. We show that we can embed $\D$ in a $k$-star decomposition of $K_{4k}$. Let $s=4k-n$ and note that $k \leq s < (6-2\sqrt{2})k$ since $k+1 \leq n < 2\sqrt{2}k$ and that $\binom{n+s}{2} \equiv 0 \mod{k}$. By Lemma~\ref{L: newconditions on s} there is a $k$-star decomposition of $L \vee K_s$ and hence the result is proved provided that  \eqref{e:sLowerBound} holds.  Now \eqref{e:sLowerBound} holds if and only if
\[\left(3k-\tfrac{1}{2}\right)^2>\left(n-\sqrt{2k}\right)^2+k(k-3)+\tfrac{1}{4}\]
and this can in turn be shown to hold using $n<2\sqrt{2}k$.

\smallskip
\noindent\textbf{Case 3:} Suppose that $1 \leq n \leq k$. Then $\D$ is empty and hence a $k$-star decomposition of $K_{2k}$, which exists by Theorem~\ref{T: TarsiandYamamoto}(a), is an embedding of $\D$.
\end{proof}

Lemma~\ref{L: even bound on s} below shows that if the constant $6-2\sqrt{2}$ in Theorem~\ref{T:NRimprovement} were decreased then the result would fail to hold for each sufficiently large $k$ that is 2 to some odd power. To see this, observe that the value of $n$ in the statement of Lemma~\ref{L: even bound on s} is at most $2\sqrt{k(2k+1)}+2\sqrt{2k}$ and hence is $2\sqrt{2}k+O(\sqrt{k})$ as $k$ becomes large.

\begin{lemma} \label{L: even bound on s}
Let $k = 2^t$ for some odd integer $t \geq 3$, let $m=\sqrt{2k}$, let $n$ be the smallest integer such that $n \equiv 0 \mod{m}$ and $n > 2\sqrt{k(2k+1)}+\sqrt{2k}$, and let $L$ be a graph of order $n$ that is a vertex disjoint union of   $\frac{n}{m}$ copies of $K_{m}$. A partial $k$-star decomposition of $K_n$ whose leave is $L$ exists and furthermore it cannot be embedded in a $k$-star decomposition of $K_{n+s}$ for any $s <  6k-n$.
\end{lemma}

\begin{proof}
Observe that $|E(L)|=\frac{n}{m}\binom{m}{2}=\frac{n}{2}(m-1)$. We first show that $L$ is the leave of a partial $k$-star decomposition of $K_n$. Note that $\deg_{\overline{L}}(y) = n-m \geq \frac{1}{2}n+k-1$ for each $y \in V(L)$ since $n > 2\sqrt{k(2k+1)}+\sqrt{2k}$. Furthermore, $|E(\overline{L})| =  \frac{n}{2}(n-m)\equiv 0 \mod{k}$ because $n \equiv 0 \mod{m}$. Therefore, by Theorem~\ref{T: Tarsi's result}, a $k$-star decomposition of $\overline{L}$ exists.

Now suppose for a contradiction that a $k$-star decomposition of $L \vee K_S$ exists where $|S|=s$ for some nonnegative integer $s < 6k-n$. We must have $n+s \equiv 0 \mod{2k}$ or $n+s \equiv 1 \mod{2k}$ by Theorem~\ref{T: TarsiandYamamoto}(b). Therefore, because $0 \leq s <6k-n$ and $n > 2k+1$, we have $s \in \{4k-n,4k-n+1\}$.

Now $\alpha(L)=\frac{n}{m}$ because an independent set in $L$ can contain at most one vertex from each copy of $K_{m}$. So we complete the proof by showing that $n+s-\tfrac{1}{k}(|E(L)|+ns+\binom{s}{2}) >\frac{n}{m}$ and hence concluding by  Lemma~\ref{l:obstacle} that there is no $k$-star decomposition of $L \vee K_{S}$. Using $|E(L)|=\frac{n}{2}(m-1)$ and $m=\sqrt{2k}$ and multiplying through by $2k$, this is equivalent to showing that
\begin{equation}\label{e:kEvenTightnessTest}
n\left(2k-2\sqrt{2k}+1\right)-s(s+2n-2k-1)
\end{equation}
is positive. Using $s \leq 4k-n+1$, \eqref{e:kEvenTightnessTest} is at least $n(n-2\sqrt{2k}) - 2k(4k+1)$. In turn this can be shown to be positive using $n > 2\sqrt{k(2k+1)}+\sqrt{2k}$.
\end{proof}

In order to prove Theorem~\ref{T:NRimprovement} when $k$ is odd, we need to make a closer examination of leaves of partial $k$-star decompositions of $K_n$ where $k < n \leq 2k$. It turns out that these leaves must contain a large clique and hence we can improve on the bound of Theorem~\ref{T:CaroWei}(b) for their independence number using Theorem~\ref{T:CaroWei}(a). Our first step is to improve on Theorem~\ref{T:CaroWei}(b) in the case where the graph considered contains a large clique.

\begin{lemma}\label{L:independent set}
If $L$ is a graph of order $n$ such that $L$ has a copy of $K_r$ as a subgraph and $|E(L)| \leq \frac{1}{2}n(r-1)$, then
\[\alpha(L) \geq 1 + \mfrac{(n-r)^2}{2|E(L)|+n-r^2}.\]
\end{lemma}

\begin{proof}
Let $V= V(L)$ and $R$ be a subset of $V$ such that $L[R]$ is a copy of $K_r$. Let $d=\frac{2|E(L)|-r(r-1)}{n-r}$ and note that $d \leq r-1$ since $|E(L)| \leq \frac{1}{2}n(r-1)$. By Theorem~\ref{T:CaroWei}(a) we have that
\begin{equation}\label{e:alphaBoundRestatement}
\alpha(L) \geq \medop\sum_{x \in V}\mfrac{1}{\deg_L(x)+1}.
\end{equation}
Observe that $\deg_L(x) \geq r-1$ for $x \in R$, that $|R|=r$, that $\sum_{x \in V}\deg_L(x)=2|E(L)|$, and that $d \leq r-1$. By convexity, the minimum value of $\sum_{i=1}^n\frac{1}{x_i+1}$, where the $x_i$ are nonnegative reals subject to the constraints $x_i \geq r-1$ for $i \in \{1,\ldots,r\}$ and $\sum_{i=1}^n x_i=2|E(L)|$, occurs when $x_i=r-1$ for each $i \in \{1,\ldots,r\}$ and $x_i=d$ for each $i \in \{r+1,\ldots,n\}$. Thus from \eqref{e:alphaBoundRestatement} we have
\[\alpha(L) \geq \mfrac{r}{(r-1)+1}+\mfrac{n-r}{d+1}=1 + \mfrac{(n-r)^2}{2|E(L)|+n-r^2}.\qedhere\]
\end{proof}

By combining Lemma~\ref{L:independent set} with Lemmas~\ref{L: no of total edges are small} and \ref{L: no of total edges are large}, we can improve on Lemma~\ref{L: newconditions on s} in the special case where $L$ is the leave of a partial $k$-star decomposition of $K_n$ and $k < n \leq 2k$. Again the $k=2$ case is covered by Lemma~\ref{L:k=2}.

\begin{lemma} \label{L: newconditions1 on s}
Let $k$, $n$ and $s$ be integers such that $s \geq k \geq 3$, $2k \geq n > k$ and $\binom{n+s}{2} \equiv 0 \mod{k}$. Any partial $k$-star decomposition of $K_n$ can be embedded in a $k$-star decomposition of $K_{n+s}$ if
\begin{equation}\label{e:sLowerBoundnew}
s> k - n +\mfrac{1}{2} + \sqrt{4k\left(\sqrt{n - k} - \tfrac{1}{\sqrt{2}}\right)^2 + k (k - 3) + \tfrac{1}{4}}
\end{equation}
\end{lemma}

\begin{proof}
Observe that the right hand side of \eqref{e:sLowerBoundnew} is real because $k \geq 3$. Suppose that  \eqref{e:sLowerBoundnew} holds. Let $\D$ be a partial $k$-star decomposition of $K_n$. We may assume that $\D$ is maximal for otherwise we can greedily add $k$-stars to $\D$ until it is maximal and then apply the proof. Let $L$ be the leave of $\D$ and note that $L$ has maximum degree at most $k-1$. Let $b = \frac{1}{k}|E(L \vee K_s)| $, note that $b$ is an integer because $|E(L \vee K_s)| \equiv 0 \mod{k}$ since $\binom{n+s}{2} \equiv 0 \mod{k}$ and $L$ is the leave of a partial $k$-star decomposition of $K_n$. If $b \geq n+s$, then a $k$-star decomposition of $L \vee K_s$ exists by Lemma~\ref{L: no of total edges are large}, so we may assume that $b < n+s$. By Lemma~\ref{L: no of total edges are small} it suffices to show that $\alpha(L) \geq n+s-b$.

Let $V_0$ be the set of vertices in $V(L)$ that have no star in $\D$ centred at them. No star in $\D$ can contain an edge between a pair of vertices in $V_0$ and hence $L[V_0]$ must be a complete graph. Because $\D$ contains $\frac{1}{k}(\binom{n}{2}-|E(L)|)$ stars, $|V_0| \geq r$ where $r= n - \frac{1}{k}(\binom{n}{2}-e)$ and $e=|E(L)|$. Note that $r \geq 1$ since $k \geq \frac{n}{2}$ from our hypotheses. So $L$ contains a copy of $K_r$ as a subgraph. Also, it follows from the definition of $r$ that $e=\binom{n}{2}-k(n-r)$ and hence, because $k \geq \frac{n}{2}$, that $e \leq \frac{1}{2}n(r-1)$. Thus, by Lemma~\ref{L:independent set} we have $\alpha(L) \geq 1+\frac{(n-r)^2}{2e+n-r^2}$.

So, because $\alpha(L)$ and $n+s-b$ are both integers, it is enough to show that $1+\frac{(n-r)^2}{2e+n-r^2} > n+s-b-1$. Using $b=\frac{1}{k}(e+ns+\binom{s}{2})$ and multiplying through by $2k$, this is equivalent to showing that
\begin{equation}\label{e:sBoundSmallDiscrim}
s^2+(2n-2k-1)s-2kn+4k+2e+\mfrac{2k(n-r)^2}{2e+n-r^2}
\end{equation}
is positive. Using $e=\binom{n}{2}-k(n-r)$, \eqref{e:sBoundSmallDiscrim} is equal to
\begin{equation}\label{e:sBoundSmallnIneq}
s^2+(2n-2k-1)s-(4k-n)(n-1)+2k\left(r+\mfrac{n-r}{n+r-2k}\right).
\end{equation}
Because $L$ contains a copy of $K_r$ as a subgraph, we have that $e \geq \binom{r}{2}$ or equivalently, using $e=\binom{n}{2}-k(n-r)$, that $\frac{1}{2}(n-r)(n+r-2k-1) \geq 0$. This implies that $2k+1-n \leq r \leq n$. Considered as a function of a real variable $r$ where $2k+1-n \leq r \leq n$, \eqref{e:sBoundSmallnIneq} is minimised when $r=2k-n+\sqrt{2n-2k}$ and, substituting this value for $r$ and rearranging, we have that \eqref{e:sBoundSmallnIneq} is at least
\[s^2 + (2n-2k-1)s-(6k-n)(n-1)+4k\left(k+\sqrt{2n-2k}-1\right).\]
Considering this last expression as a quadratic in $s$, we can see that it is positive when  \eqref{e:sLowerBoundnew} holds. Thus \eqref{e:sBoundSmallDiscrim} is positive and $\alpha(L) \geq n+s-b$, as required.
\end{proof}

We now finish the proof of Theorem~\ref{T:NRimprovement} by considering the case where $k$ is odd.

\begin{proof}[\textbf{\textup{Proof of Theorem~\ref{T:NRimprovement}.}}]
When $k$ is even the result follows from Lemma~\ref{L:NRimprovementEven}, so we may assume that $k$ is odd. Let $\D$ be a partial $k$-star decomposition of $K_n$ and $L$ be its leave.  Note that we will have $|E(L \vee K_s)| \equiv 0 \mod{k}$ for any integer $s$ such that $n+s \equiv 0 \mod{k}$. We consider four cases according to the value of $n$.

\noindent\textbf{Case 1:} Suppose that $n \geq 2\sqrt{2}k$. Let $s$ be an integer such that $\frac{5}{4}k \leq s < \frac{9}{4}k$ and $n+s \equiv 0 \mod{k}$. We saw in Case 1 of the proof of Lemma~\ref{L:NRimprovementEven} that the right hand side of \eqref{e:sLowerBound} is less than $(4-2\sqrt{2})k$ when $n \geq 2\sqrt{2}k$. So by Lemma~\ref{L: newconditions on s} there is a $k$-star decomposition of $L \vee K_s$ and hence the result is proved, because $s \geq \frac{5}{4}k >(4-2\sqrt{2})k$.

\smallskip
\noindent\textbf{Case 2:} Suppose that $\frac{7}{4}k < n < 2\sqrt{2}k$. We show that we can embed $\D$ in a $k$-star decomposition of $K_{4k}$. Let $s=4k-n$ and note that $k \leq s < \frac{9}{4}k$ since $\frac{7}{4}k < n < 2\sqrt{2}k$ and that $\binom{n+s}{2} \equiv 0 \mod{k}$. We showed in Case 2 of the proof of Lemma~\ref{L:NRimprovementEven} that \eqref{e:sLowerBound} holds when $s=4k-n$ and $n < 2\sqrt{2}k$. So by Lemma~\ref{L: newconditions on s} there is a $k$-star decomposition of $L \vee K_s$.

\smallskip
\noindent\textbf{Case 3:} Suppose that $k+1 \leq n \leq \frac{7}{4}k$. We show that we can embed $\D$ in a $k$-star decomposition of $K_{3k}$. Let $s=3k-n$ and note that $k \leq s < \frac{9}{4}k$ since $k+1 \leq n \leq \frac{7}{4}k$ and that $\binom{n+s}{2} \equiv 0 \mod{k}$. Then \eqref{e:sLowerBoundnew} holds if and only if
\begin{equation}\label{e:NRimprovementProofCase3}
\left(2k-\tfrac{1}{2}\right)^2>4k\left(n-\sqrt{2n-2k}\right)-k(3k+1)+\tfrac{1}{4}.
\end{equation}
For $n \geq k+1$, the right hand side of \eqref{e:NRimprovementProofCase3} is increasing in $n$ and hence \eqref{e:NRimprovementProofCase3} can be shown to hold for $k+1 \leq n \leq \frac{7}{4}k$ by substituting $n=\frac{7}{4}k$. So by Lemma~\ref{L: newconditions1 on s} there is a $k$-star decomposition of $L \vee K_s$.

\smallskip
\noindent\textbf{Case 4:} Suppose that $1 \leq n \leq k$. Then $\D$ is empty and hence a $k$-star decomposition of $K_{2k}$, which exists by Theorem~\ref{T: TarsiandYamamoto}(a), is an embedding of $\D$.
\end{proof}

Finally, we prove Lemma~\ref{L: odd bound on s}, which shows that if the constant $\frac{9}{4}$ in Theorem~\ref{T:NRimprovement} were decreased then the result would fail to hold for each sufficiently large $k$ that is a power of an odd prime. To see this, observe that the definition of $n$ in the statement of Lemma~\ref{L: odd bound on s} can be rephrased as $n=\frac{1}{2}a+k$ where $a$ is the smallest even perfect square that is greater than $\frac{3}{2}k+\sqrt{6k+6}+\frac{5}{2}$. Clearly then, $a=\frac{3}{2}k+O(\sqrt{k})$ and hence $n=\frac{7}{4}k+O(\sqrt{k})$ as $k$ becomes large.

\begin{lemma} \label{L: odd bound on s}
Let $k$ be a sufficiently large integer that is a power of an odd prime and let $n$ be the smallest integer such that $n > \frac{7}{4}k+\frac{1}{2}\sqrt{6k+6}+\frac{5}{4}$ and $\sqrt{2n-2k}$ is an integer. Let $m = \sqrt{2n-2k}$ and $r = 2k-n+m$, and let $L$ be a graph of order $n$ that is a vertex disjoint union of  $m-1$ copies of $K_m$ and a copy of $K_r$. A partial $k$-star decomposition of $K_n$ whose leave is $L$ exists and furthermore it has no embedding in a $k$-star decomposition of $K_{n+s}$ for any $s < 4k-n$.
\end{lemma}

\begin{proof}
Observe that, for sufficiently large $k$, $r=\frac{k}{4}+ O(\sqrt{k})$ because $n=\frac{7}{4}k+O(\sqrt{k})$ as noted in the paragraph before the lemma. We first show that $L$ is the leave of a partial $k$-star decomposition. Let $V_0$ be the vertex set of the copy of $K_r$ in $L$ and let $V_1,\ldots,V_{m-1}$ be the vertex sets of the copies of $K_m$ in $L$. Let $\gamma: V(L) \rightarrow \Z^{\geq 0}$ be defined by $\gamma(x) = 0$ for each $x \in V_0$ and $\gamma(y) = 1$ for each $y \in V(L) \setminus V_0$. Then $\gamma$ is a precentral function for $\overline{L}$, because we have $\frac{1}{k}(\binom{n}{2}-|E(L)|) = m(m-1)$ using $|E(L)|=\binom{r}{2}+(m-1)\binom{m}{2}$, the definition of $r$ and $n=\frac{1}{2}m^2+k$. Let $\G$ be  $\overline{L}$ equipped with $\gamma$ and let $T \in \T$. We will show that $\Delta_T=0$ and hence that a $k$-star decomposition of $\overline{L}$ exists. For each $i \in \{1,\ldots,m-1\}$, we have $V_i \subseteq T$ or $T \cap V_i = \emptyset$ by Lemma~\ref{L:about S' in G} with $U=V_i$. So without loss of generality we can assume that $T = V_1 \cup  \cdots \cup V_t$ for some $t \in \{0,\ldots,m-1\}$. Then $\Delta^+_T=\binom{t}{2}m^2+mt(n-mt)$ and $\Delta^-_T = kmt$. Thus, using  $n=\frac{1}{2}m^2+k$ and simplifying,
\[\Delta_T=\tfrac{1}{2}tm^2(m-1-t)\]
which is nonnegative since $t \in \{0,\ldots,m-1\}$. Thus $\Delta_T=0$ and a $k$-star decomposition of $\overline{L}$ exists.

Now suppose for a contradiction that a $k$-star decomposition of $L \vee K_S$ exists where $|S|=s$ for some nonnegative integer $s < 4k-n$. We must have $n+s \equiv 0 \mod{k}$ or $n+s \equiv 1 \mod{k}$ by Theorem~\ref{T: TarsiandYamamoto}(c). Therefore, because $0 \leq s <4k-n$ and $n > k+1$, we have $s \in \{2k-n,2k-n+1,3k-n,3k-n+1\}$.

Now $\alpha(L)=m$ because an independent set in $L$ can contain at most one vertex from the copy of $K_r$ and at most one vertex from each copy of $K_m$. So we complete the proof by showing that $n+s-\tfrac{1}{k}(|E(L)|+ns +\binom{s}{2})>m$ and hence concluding by  Lemma~\ref{l:obstacle} that there is no $k$-star decomposition of $L \vee K_s$. Using $|E(L)|=\binom{r}{2}+(m-1)\binom{m}{2}$, the definitions of $r$ and $m$, and multiplying through by $2k$, this is equivalent to showing that
\begin{equation}\label{e:kOddTightnessTest}
n(6k-n+1)-4k\left(k+\sqrt{2n-2k}\right)-s(s+2n-2k-1)
\end{equation}
is positive. Using $s \leq 3k-n+1$, \eqref{e:kOddTightnessTest} is at least $k(4n-7k-4\sqrt{2n-2k}-1)$. In turn this can be shown to be positive using $n > \frac{7}{4}k+\frac{1}{2}\sqrt{6k+6}+\frac{5}{4}$.
\end{proof}

\bigskip
\noindent\textbf{Acknowledgments.}
The second author was supported by Australian Research Council grants DP150100506 and FT160100048.

%-----------------------------------------------------------------------------

%-----------------------------------------------------------------------------

\end{document}